\documentclass[12pt]{amsart}
\usepackage{geometry}                
\usepackage{amssymb}
\usepackage{ulem}
\usepackage[dvipsnames]{xcolor}

\usepackage{hyperref}

\usepackage{tikz}
\usetikzlibrary{shapes}

\newtheorem{theorem}{Theorem}[section]
\theoremstyle{plain}
\newtheorem{corollary}[theorem]{Corollary}

\newtheorem{example}[theorem]{Example}

\newtheorem{prop}[theorem]{Proposition}
\newtheorem{remark}[theorem]{Remark}

\def\sone#1#2{\begin{bmatrix}#1 \\ #2 \end{bmatrix}}

\def\ua{\!\!\uparrow}
\def\da{\!\!\downarrow}

\def\GS{\text{GS}}
\def\nua{\text{nua}}
\def\nda{\text{nda}}

\def\gf{A}

\def\apoln#1{A_#1(x)}     
\def\anum#1{A_#1}

\def\asc{\text{asc}}
\def\gfp#1{\gf_#1(x;u,d)}
\def\expgf{\gf(t,x;u,d)}

\def\nn{\nonumber}
\def\fracpar#1#2{\frac{\partial #1}{\partial #2}}

\tikzset{
  dot hidden/.style={},
  line hidden/.style={},
  dot colour/.style={dot hidden/.append style={color=#1}},
  dot colour/.default=black,
  line colour/.style={line hidden/.append style={color=#1}},
  line colour/.default=black
}

\usepackage{xparse}

\NewDocumentCommand{\domino}{mm}{
\begin{tikzpicture}[x=2em,y=2em,radius=0.1]
\draw[rounded corners=0.5,line hidden] (0,0) rectangle (1,2);

\draw[line hidden] (0,1) -- (1,1);
\ifodd#1
  \fill[dot hidden] (0.5,1.5) circle;
\fi
\ifnum#1>1
   \fill[dot hidden] (0.2,1.2) circle;
   \fill[dot hidden] (0.8,1.8) circle; 
   \ifnum#1>3
      \fill[dot hidden] (0.8,1.2) circle;
      \fill[dot hidden] (0.2,1.8) circle;   
   \fi
   \ifnum#1>5
      \fill[dot hidden] (0.2,1.5) circle;
      \fill[dot hidden] (0.8,1.5) circle;  
   \fi
\fi
\ifodd#2
   \fill[dot hidden] (0.5,0.5) circle;
\fi
\ifnum#2>1
   \fill[dot hidden] (0.2,0.2) circle;
   \fill[dot hidden] (0.8,0.8) circle;
   \ifnum#2>3
      \fill[dot hidden] (0.8,0.2) circle;
      \fill[dot hidden] (0.2,0.8) circle;   
   \fi
   \ifnum#2>5
      \fill[dot hidden] (0.2,0.5) circle;
      \fill[dot hidden] (0.8,0.5) circle;  
   \fi
\fi
\end{tikzpicture}
}
\NewDocumentCommand{\umino}{m}{
\begin{tikzpicture}[x=2em,y=2em,radius=0.1]
\draw[rounded corners=0.5,line hidden] (0,0) rectangle (1.4,1.4);
\node (A) at (0.2, 0.1){};
\node (B) at (0.2, 1.3){};
\draw [->,very thick] (A) -- (B);
\ifodd#1
   \fill[dot hidden] (0.8,0.7) circle;
\fi
\ifnum#1>1
   \fill[dot hidden] (0.5,0.4) circle;
   \fill[dot hidden] (1.1,1.0) circle;
   \ifnum#1>3
      \fill[dot hidden] (1.1,0.4) circle;
      \fill[dot hidden] (0.5,1.0) circle;   
   \fi
   \ifnum#1>5
      \fill[dot hidden] (0.5,0.7) circle;
      \fill[dot hidden] (1.1,0.7) circle;  
   \fi
\fi
\end{tikzpicture}
}
\NewDocumentCommand{\mino}{m}{
\begin{tikzpicture}[x=2em,y=2em,radius=0.1]
\draw[rounded corners=0.5,line hidden] (0,0) rectangle (1.0,1.0);

\ifodd#1
   \fill[dot hidden] (0.5,0.5) circle;
\fi
\ifnum#1>1
   \fill[dot hidden] (0.2,0.2) circle;
   \fill[dot hidden] (0.8,0.8) circle;
   \ifnum#1>3
      \fill[dot hidden] (0.8,0.2) circle;
      \fill[dot hidden] (0.2,0.8) circle;   
   \fi
   \ifnum#1>5
      \fill[dot hidden] (0.2,0.5) circle;
      \fill[dot hidden] (0.8,0.5) circle;  
   \fi
\fi
\end{tikzpicture}
}

\begin{document}
\title[Generalized permutations]{Generalized permutations related to the degenerate Eulerian numbers}

\author[O.~Herscovici]{Orli Herscovici}
\address{O.~Herscovici\\
Department of Mathematics, University of Haifa, 3498838  Haifa,
Israel}
\email{orli.herscovici@gmail.com}
\maketitle

\begin{abstract}
In this work we propose a combinatorial model that generalizes the standard definition of permutation. Our model generalizes the degenerate Eulerian polynomials and numbers of Carlitz from 1979 and provides  missing combinatorial proofs for some relations on the 
degenerate Eulerian numbers.
\end{abstract}

\bigskip

\smallskip


\section{Motivation}
Permutation is one of the widely studied combinatorial objects. 
A common definition of a permutation is an arranged sequence of  
$n$ different objects: for example the elements of the set 
$[n]=\{1,2,\ldots,n\}$. Given a permutation $\pi=\pi_1\pi_2\cdots\pi_n$, 
then any index $1\leq k\leq n$ is either an ascent or a descent 
of the permutation $\pi$. If $\pi_k<\pi_{k+1}$ then $k$ is an ascent, 
otherwise it is a descent of $\pi$ (i.e. $\pi_k>\pi_{k+1}$). The last index of the permutation is neither ascent or descent. Let us 
denote by $\anum{{n,k}}$ a number of permutations of a length $n$ 
with exactly $k$ ascents, then $\anum{{n,k}}$ are the Eulerian 
numbers and they are the coefficients of the Eulerian polynomials 
$\apoln{n}$,  i.e. $\apoln{n}=\sum_{k=0}^n\anum{{n,k}}x^k$.  
The Eulerian polynomials $\apoln{n}$ are given by the following  
exponential function
\begin{align}
\frac{x-1}{x-e^{t(x-1)}}=\sum_{n\geq 0}\apoln{n}\frac{t^n}{n!}. \label{p12GF:Eulerian}
\end{align}
The Eulerian numbers satisfy the recurrence relation 
\begin{align}
\anum{{n,k}}=(k+1)\anum{{n-1,k}}+(n-k)\anum{{n-1,k-1}}, \label{p12:rec1}
\end{align}
and
\begin{align}
\sum_{k=0}^{n-1}\anum{{n,k}}=n!\label{p12:rec2}
\end{align}

In 1954 Carlitz proposed a $q$-Eulerian polynomials \cite{Carlitz1954}, which are a generalization of the Eulerian polynomials based on $q$-exponential function and $q$-calculus \cite{Kac2002}. Later, in 1975, he show that this generalization is related to two statistics on permutations: the number of descents and the major index, which is a  sum of all descent indices \cite{Carlitz1975}.  

In his work \cite{Carlitz1979} from 1979 Carlitz defined degenerate Eulerian numbers $\anum{{n,k}}(\lambda)$ as
\begin{align}
\frac{1-x}{1-x(1+\lambda t(1-x))^\mu}=1+\sum_{n=1}^\infty\sum_{k=1}^n\anum{{n,k}}(\lambda)x^k\frac{t^n}{n!}. \label{p12GF:degEulPol}
\end{align}
The degenerate Eulerian numbers satisfy a very similar to \eqref{p12:rec2} equation
\begin{align}
\sum_{k=0}^{n-1}\anum{{n,k}}(\lambda)=n!\label{p12:rec3}
\end{align}
and a recurrence relation 
\begin{align}
\anum{{n,k}}(\lambda)=(k+1-(n-1)\lambda)\anum{{n-1,k}}+(n-k+(n-1)\lambda)\anum{{n-1,k-1}}. \label{p12:rec4}
\end{align}
Both results \eqref{p12:rec3} and \eqref{p12:rec4} can be proved analytically by using the generating function \eqref{p12GF:degEulPol}, however there is neither combinatorial interpretation of the degenerate Eulerian numbers $\anum{{n,k}}(\lambda)$ nor combinatorial proof of these equations.

In this work we propose a combinatorial model that generalizes the standard definition of permutation. Our model not only provides a combinatorial proof for \eqref{p12:rec3} -\eqref{p12:rec4}, but also it generalizes the degenerate Eulerian polynomials and numbers of Carlitz.

\section{New extension of permutations}
Let us consider a permutation of $[n]$ as a permutation of tiles with dots, similar to domino but containing one single square: \resizebox{0.5cm}{0.4cm}{%
\mino{1}
}, 
\resizebox{0.5cm}{0.4cm}{%
\mino{2}
}, 
\resizebox{0.5cm}{0.4cm}{%
\mino{3}
}, \ldots. 
Each one of such tiles is invariant under rotation by $180^\circ$. 
Let us consider a recursive construction of permutation starting from 
a permutation containing a single tile \resizebox{0.5cm}{0.4cm} {\mino{1}}. 
Assuming given a permutation $\pi$ of a length $n-1$, a permutation $\pi'$ 
of a length $n$ can be constructed from $\pi$ by inserting a tile with 
$n$ dots. Such insertion  not necessarily adds an ascent into the
permutation. Let us now at each step to add another tile with a special property -- this tile inserted into permutation will add an ascent. We depict this tile as 
\resizebox{0.5cm}{0.4cm}{\umino{2}}. Obviously, the rotation by $180^\circ$ turns it into
\resizebox{0.5cm}{0.4cm}{\rotatebox[origin=c]{180}{\umino{2}}}.

These tiles are different under rotation, so we define different actions 
for these two positions. From hereon we denote a tile with upward 
arrow by number of dots followed by a symbol ``$\uparrow$'' and a 
tile with downward arrow by number  of dots followed by a symbol 
``$\downarrow$''. For example, the tiles above will be denoted by 
$2\!\uparrow$ and $2\!\downarrow$, respectively, while a tile without 
arrows will be denoted by the number of dots on it. Therefore at the 
step $n$ of recursive construction of permutation we add one of 
three possible tiles: regular $n$ or $n\!\!\uparrow$ or its rotation $n\!
\!\downarrow$. Let us denote by $\asc(\pi)$ a set of ascents of the permutation $\pi=\pi_1\pi_2\cdots\pi_{n-1}$, i.e.
\begin{align}
\asc(\pi)=\{\,i\, | \, \pi_i<\pi_{i+1}\},\label{p12:DEFasc}
\end{align}
and by $|\asc(\pi)|$ its cardinality.
Now we need to define the rules for tiles insertion. We formulate them as following.
\begin{enumerate}
\item[\textbf{Rule1:}]   insertion $n$ in the first position does not change the number of ascents 
\begin{align*}
|\asc(\pi')|=|\asc(\pi)|, \quad\text{and}\quad \asc(\pi')=\{k\,|\,k-1\in\asc(\pi)\}.
\end{align*}
\item[\textbf{Rule2:}]  insertion $n$ in the last position will always increase the number of ascents by producing a new ascent at the position $n-1$
\begin{align*}
|\asc(\pi')|=|\asc(\pi)|+1, \quad\text{and}\quad \asc(\pi')=\asc(\pi)\cup\{n-1\}.
\end{align*}
\item[\textbf{Rule3:}]  insertion $n$ in position $i$ for $2\leq i\leq n$  
will cause to that ascents before on left from $n$ will stay without change and ascents on the right from them will increase by 1. Moreover,
\begin{enumerate}
\item if $i-1$ was an ascent then it is still an ascent 
\begin{align*}
|\asc(\pi')|&=|\asc(\pi)|, \\ \text{and}\quad \asc(\pi')&=\{k<i\,|\,k\in\asc(\pi)\}\cup
\{k+1\,|\,k\in\asc(\pi),\,k\geq i\}.
\end{align*}
\item  if $i-1$ was not an ascent then it turns to an ascent 
\begin{align*}
|\asc(\pi')|&=|\asc(\pi)|+1, \\ \quad\quad\text{and}\quad \asc(\pi')&=\{k<i\,|\,k\in\asc(\pi)\}\cup\{i\}\cup\{k+1\,|\,k\in\asc(\pi),\,k\geq i\}.
\end{align*}
\end{enumerate}
\item[\textbf{Rule4:}]   insertion $n\ua$ in position $i<n$ causes that $i$ is an ascent for all $1\leq i\leq n-1$
\begin{align*}
|\asc(\pi')|&=|\asc(\pi)|+1, \\ \quad\quad\text{and}\quad \asc(\pi')&=\{k<i\,|\,k\in\asc(\pi)\}\cup\{i\}\cup\{k+1\,|\,k\in\asc(\pi),\,k\geq i\}.
\end{align*}
Note, it means that a tile $n\ua$ can be inserted on the \emph{left} from any of $n-1$ tiles of permutation.

\item[\textbf{Rule5}]   insertion $n\da$ in position $i<n$ will not change existing ascents of permutations
\begin{align*}
|\asc(\pi')|&=|\asc(\pi)|, \\ \quad\text{and}\quad \asc(\pi')&=\{k<i\,|\,k\in\asc(\pi)\}\cup\{k+1\,|\,k\in\asc(\pi),\,k\geq i\}.
\end{align*}
Tile $n\da$ can be inserted on the \emph{left} from any of $n-1$ tiles of permutation.
\end{enumerate}
We call a permutation $\pi$ {\it{generalized}} if it can be obtained from a trivial permutation of $[1]$ recursively by applying rules Rule1--Rule5.  
\begin{example}
Let us consider a recursive construction of a  permutation $\pi=2\!\!\downarrow,4,3\!\!\uparrow,1$. 
\begin{itemize}
\item[Step 0:] a trivial permutation $\pi =1$ with no ascents.
\item[Step 1:] insertion of a tile $2\!\!\downarrow$ on the left from $1$, so that $\pi =2\!\!\downarrow,1$ and w.r.t the Rule 5 the number of ascents is not changed.
\item[Step 2:] insertion of a tile $3\!\!\uparrow$ on the left from $1$, so that $\pi=2\!\!\downarrow,3\!\!\uparrow,1$. In accordance with the Rule4 we obtain that $\asc(\pi)=\{2\}$ and, respectively, $|\asc(\pi)|=1$.
\item[Step 4:] insertion of a tile $4$ so that $\pi=2\!\!\downarrow,4,3\!\!\uparrow,1$. In accordance with the Rule3(b) we obtain that $\asc(\pi)=\{1,3\}$ and, respectively, $|\asc(\pi)|=2$.
\end{itemize}
One can see that the permutation $\pi=2\!\!\downarrow,4,3\!\!\uparrow,1$ can be obtained by  Rule1-Rule5 and, thus, it is a generalized permutation.

\begin{remark}
Given a generalized permutation $\pi=\pi_1\pi_2\cdots\pi_n$ such that $\pi_i=k\!\!\uparrow$ or $\pi_i=k\!\!\downarrow$ for some $2\leq k\leq n$. It follows from the Rule1-Rule5 that there exist an index $j>i$ and integer $m<k$ such that the $\pi_j=m$.  
\end{remark}

\begin{example}
Let us consider another  permutation $\pi=2\!\!\downarrow,1,3\!\!\uparrow,4$. This is not a generalized permutation w.r.t. the  Rule1-Rule5. We can check by recursive construction. The first two steps are identical to the Step 0 and Step 1 from the previous example. However, the insertion of the tile $3\!\!\uparrow$ is incorrect w.r.t. the Rule4. 
\end{example}
\end{example}

Let us denote by $\GS_{n}$ the set of all generalized permutations of the set $[n]$,  
and by $|\GS_n|$ its cardinality.
\begin{prop} For any generalized permutation $\pi\in\GS_n$, we have $0\leq|\asc(\pi)|<n$.
\end{prop}
\begin{proof} It follows immediately from the rules Rule1-Rule5.
\end{proof}
\begin{prop} A number of generalized permutations $\GS_n$ of 
a set $[n]$ is given by a following generating function
\begin{align}
\frac{1}{\sqrt[3]{(1-3t)}}=\sum_{n=0}^\infty|\GS_n|\frac{t^n}{n!},\label{p12:eq1}
\end{align}
and
\begin{align}
|\GS_n|=\prod_{k=1}^n(3k-2)=\sum_{j=0}^{n-1}3^{j}\cdot|s(n,n-j)|,\label{p12:eq2}
\end{align}
where $s(n,j)$ are the Stirling numbers of the first kind.

\end{prop}
\begin{proof}
Let us start from the first part of the equation \eqref{p12:eq2}. 
The proof is by induction on $n$. There is only one permutation of 
$[1]$, i.e. $|\GS_1|=1$ and the statement holds.
Assuming it holds for $n-1$ for $n\geq 2$, we will prove it for $n$. 
Let $\pi=\pi_1\pi_2\ldots\pi_{n-1}\in\GS_{n-1}$. We construct a 
 permutation of length $n$ by inserting one of tiles $n$, $n\ua$, or 
$n\da$ into $\pi$. The tile $n$ might be inserted before or after 
any $\pi_i$, which means $n$ possibilities. The tile $n\ua$ might 
be inserted only on the left from $\pi_i$, which means $n-1$ 
possibilities. In the same way we have $n-1$ possibilities for 
$n\da$. Gathering all cases together, we obtain $|\GS_n|=|
\GS_{n-1}|\cdot(3n-2)$. Recursive substitution gives that $|\GS_n|
=\prod_{k=1}^n(3k-2)$.
 
Note that $\prod_{k=1}^n(3k-2)=\prod_{k=0}^{n-1}(3k+1)$. It was 
shown in \cite{Herscovici2017} that $\prod_{k=0}^{n-1}
(ku+1)=\sum_{j=1}^n|s(n,j)|u^{n-j}$, where $s(n,j)$ are the Stirling 
numbers of the first kind, and in \cite{Herscovici2019} that
$(1-ut)^{-\frac{1}{u}}=\sum_{n\geq 0}\sum_{j=0}^n|s(n,n-j)|u^{j}\frac{t^n}{n}$. Therefore, substitution $u=3$ completes the proof.
\end{proof}

In the next section we obtain a recurrence relation for the generalized permutations.

\section{Recurrence relation}
Let $\pi=\pi_1\pi_2\ldots\pi_n\in\GS_n$ be a generalized permutation of $[n]$. We denote  by $\nua(\pi)$ a number of tiles with $\,\,\ua$ in $\pi$ and by $\nda(\pi)$ a number of tiles with $\,\,\da$ in $\pi$.
We denote by $\gf(n,k;u,d)$ a generating function for the generalized permutations of $[n]$ with exactly $k$ ascents according to the statistics $\nua(\pi)$ and $\nda(\pi)$, that is
\begin{align}
\gf(n,k;u,d)=\sum_{\pi\in\GS_n}u^{\nua(\pi)}d^{\nda(\pi)}.
\end{align}
It follows immediately from this definition that 
\begin{align}
\gf(n,k;u,d)=\left\{\begin{array}{cl}\neq 0 & \text{for $0\leq k<n$ for all $n\in\mathbb{N}$}, \\ =0 & \text{otherwise}.\end{array}\right.
\end{align}
\begin{theorem}\label{p9::thm1}
The generating function $\gf(n,k;u,d)$ 
satisfies the following recurrence relation
\begin{align}
\gf(n,k;u,d)&=(k+1+(n-1)d)\gf(n-1,k;u,d)\nn\\
&+(n-k+(n-1)u)\gf(n-1,k-1;u,d),\label{p12:AnkRecRel}
\end{align}
with initial condition $\gf(1,0;u,d)=1$.
\end{theorem}
\begin{proof}
From the Rule1--Rule5 for recursive construction of generalized permutations, it follows that a generalized permutation of $[n]$ with exactly $k$ ascents can be obtained  from a generalized permutation of $[n-1]$ either with $k-1$ ascents or with $k$ ascents. If a permutation of $[n-1]$ has $k$ ascents than the number of ascents should not be increased. It can be achieved in two ways, namely, by adding either $n$  or $n\da$. The adding $n$ will not increase ascents if it will be inserted either in the first position or after any of existing $k$ ascents. There are $(k+1)\gf(n-1,k;u,d)$ possibilities. The tile $n\da$ might be inserted on the left from any element of permutation of $[n-1]$. Each such insertion will increase the statistic \nda, and, therefore there are $(n-1)d\gf(n-1,k;u,d)$ possibilities.

If a permutation of $[n-1]$ has $k-1$ ascents, then the number of ascents should be increased. It can be achieved either by inserting $n$ or by inserting $n\ua$. The insertion of $n\ua$ will increase the number of ascents while it is inserted on the left from any of $n-1$ tiles. Each such insertion increases the statistic $\nua$. Therefore there are $(n-1)u\gf(n-1,k-1;u,d)$ possibilities. The insertion of $n$ will increase the number of ascents only if $n$ is inserted after non-ascent position. There are $(n-k)\gf(n-1,k-1;u,d)$ possibilities to do it.
Summarising over these four cases completes the proof.
\end{proof}
\begin{example}
By applying the Theorem~\ref{p9::thm1}, we obtain the following explicit expressions for the first values of $\gf(n,k;u,d)$.
\begin{align}
\begin{tabular}{lll}
$\gf(1,0)=1$ & & \\
$\gf(2,0)=1+d$ & $\gf(2,1)=1+u$ &\\
$\gf(3,0)=1+3d+2d^2$ & $\gf(3,1)=4+4d+4u+4du$ & $\gf(3,2)=1+3u+2u^2$\\
\end{tabular}
\end{align}
\end{example}
\begin{corollary}
For all integers $n$
\begin{align}
\sum_{k=0}^{n-1}\gf(n,k;1,1)=|\GS_n|.
\end{align}
\end{corollary}
\begin{corollary} \label{p9::cor1}
The ordinary generating function for generalized permutations of $[n]$ with no ascents has the following form
\begin{align}
\gf(n,0;u,d)=\sum_{j=0}^{n}\sone{n}{n-j}d^j,
\end{align}
while its exponential generating function is given by
\begin{align}
(1-d\cdot t)^{-\frac{1}{d}}=\sum_{n=0}^\infty\gf(n,0;u,d)\frac{t^n}{n!}.\label{p9::gf1}
\end{align}
\end{corollary}
\begin{proof}
It follows immediately from the Theorem~\ref{p9::thm1}, that
\begin{align}
\gf(n,0;u,d)&=(1+(n-1)d)\gf(n-1,0;u,d).
\end{align}
Recursive substitution leads to $\gf(n,0;u,d)=\displaystyle{\prod_{j=0}^{n-1}(jd+1)}$, which can be expanded as $\gf(n,0;u,d)=\displaystyle{\sum_{j=0}^{n}\sone{n}{n-j}d^j}$ (see \cite{Herscovici2017}). It was shown in  \cite{Herscovici2019} that this sum is the coefficient of $\dfrac{t^n}{n!}$ in expansion into Taylor series of the function $(1-d\cdot t)^{-\frac{1}{d}}$, and the proof is complete.
\end{proof}
\begin{corollary} The ordinary generating function $\gf(n,n-1;u,d)$ for generalized permutations of $[n]$ with maximal number of ascents has the following form
\begin{align}
\gf(n,n-1;u,d)=\sum_{j=0}^{n}\sone{n}{n-j}u^j,
\end{align}
while its exponential generating function is given by
\begin{align}
(1-u\cdot t)^{-\frac{1}{u}}=\sum_{n=0}^\infty\gf(n,n-1;u,d)\frac{t^n}{n!}.
\end{align}
\end{corollary}
\begin{proof}It follows immediately from the Theorem~\ref{p9::thm1}, that
\begin{align}
\gf(n,n-1;u,d)&=(1+(n-1)u)\gf(n-1,n-2;u,d).
\end{align}
Recursive substitution leads to $\gf(n,n-1;u,d)=\displaystyle{\prod_{j=0}^{n-1}(ju+1)}$, and the rest is similar to the proof of the Corollary~\ref{p9::cor1}.
\end{proof}

\begin{remark}
Obviously, there is a simple bijection between generalized permutations of $[n]$ with no ascents and generalized permutations of $[n]$ with $n-1$ ascents (maximal number of ascents in a permutation of $[n]$).
\end{remark}
\begin{theorem}[Connection to the degenerate Eulerian numbers] 
For all integers $n,k$, the degenerate Eulerian numbers $\anum{{n,k}}(\lambda)$ defined by the generating function \eqref{p12GF:degEulPol} satisfy
\begin{align*}
\anum{{n,k}}(\lambda)=\gf(n,k;\lambda,-\lambda).
\end{align*}
\end{theorem}
\begin{corollary}[Combinatorial proof of \eqref{p12:rec3}]
For all integers $n$
\begin{align}
\sum_{k=0}^{n-1}\gf(n,k;\lambda,-\lambda)=n!
\end{align}
\end{corollary}
\begin{proof}
Let us assign to the tiles with upward arrow $\lambda$ and to the tiles with downward arrow $-\lambda$ for all generalized permutations of a length $n$. Now summing over all coefficients will be equivalent to summing over all permutations without tiles containing arrows -- standard permutations of a length $n$. There are exactly $n!$ permutations without arrows, and the proof is complete.
\end{proof}

\section{Generalized  polynomials and their generating function}
Let us define a generalization of the degenerate Eulerian polynomials $\apoln{n}$ as
\begin{align}
\gfp{n}=\sum_{k=0}^{n-1}\gf(n,k;u,d)x^k.\label{p12:GF-An}
\end{align}
\begin{theorem}
For all positive integers $n$, the polynomials $\gf_n(x;u,d)$ satisfy the following recurrence relation
\begin{align}
\gfp{n}=\left[1+(n-1)\big(d+(1+u)x\big)\right]\gfp{{n-1}}+(x-x^2)\fracpar{}{x}\gfp{{n-1}}, \label{p12:AnRecRel}
\end{align}
with initial conditions $\gfp{0}=1$ and $\gfp{n}=0$ for negative $n$.
\end{theorem}
\begin{proof}
Multiplying \eqref{p12:AnkRecRel} by $x^{k-1}$ and summing over all $k\geq 1$ we obtain
\begin{align}
\sum_{k\geq 1}\gf(n,k;u,d)x^k&=\sum_{k\geq 1}(k+1)\gf(n-1,k;u,d)x^k\nn\\
&+(n-1)d\sum_{k\geq 1}\gf(n-1,k;u,d)x^k\nn\\
&+(n-k)\sum_{k\geq 1}\gf(n-1,k-1;u,d)x^k\nn\\
&+(n-1)u\sum_{k\geq 1}\gf(n-1,k-1;u,d)x^k.
\end{align}
By using the definition \eqref{p12:GF-An}, the last equation can be written as
\begin{align}
\gfp{n}-\gf(n,0;u,d)&=\fracpar{}{x}\left[x(\gfp{{n-1}}-\gf(n-1,0;u,d))\right]\nn\\
&+(n-1)d\cdot\big(\gfp{{n-1}}-\gf(n-1,0;u,d)\big)\nn\\
&+nx\gfp{{n-1}}-x\fracpar{}{x}\big(x\gfp{{n-1}}\big)\nn\\
&+(n-1)ux\gfp{{n-1}}.
\end{align}
Differentiating w.r.t. to $x$ and collecting coefficients of $\gfp{{n-1}}$ and $\fracpar{}{x}\gfp{{n-1}}$ lead to
\begin{align}
\gfp{n}&=\big[1+(n-1)d+nx-x+(n-1)ux\big]\gfp{{n-1}}\nn\\
&+\big[x-x^2\big]\fracpar{}{x}\gfp{{n-1}}\nn\\
&+\gf(n,0;u,d)-\big[1+(n-1)d\big]\gf(n-1,0;u,d).
\end{align}
It follows from \eqref{p12:AnkRecRel} that $\gf(n,0;u,d)-\big[1+(n-1)d\big]\gf(n-1,0;u,d)=0$. Finally, simplification of the coefficient of $\gfp{{n-1}}$ completes the proof.
\end{proof}
Let us denote by $\expgf$ an exponential generating function for the polynomials $\gfp{n}$, i.e.
\begin{align}
\expgf=\sum_{n\geq 0}\gfp{n}\frac{t^n}{n!}.\label{p12:GF-A}
\end{align}
\begin{theorem}
The generating function 
$\expgf$ satisfies a following differential equation
\begin{align}
\expgf=-(x-x^2)\fracpar{}{x}\expgf+\left[1-\big((1+u)x+d)t\big)\right]\fracpar{}{t}\expgf.
\end{align}
\end{theorem}
\begin{proof}
Let us multiply \eqref{p12:AnRecRel} by $\frac{t^{n-1}}{(n-1)!}$ and sum over $n\geq 1$. We obtain
\begin{align}
\sum_{n\geq1}\gfp{n}\frac{t^{n-1}}{(n-1)!}&=\sum_{n\geq 1}\gfp{{n-1}}\frac{t^{n-1}}{(n-1)!}\nn\\
&+\big[(1+u)x+d\big]\sum_{n\geq 1}(n-1)\gfp{{n-1}}\frac{t^{n-1}}{(n-1)!}\nn\\
&+(x-x^2)\fracpar{}{x}\sum_{n\geq1}\gfp{{n-1}}\frac{t^{n-1}}{(n-1)!}.
\end{align}
By using \eqref{p12:GF-A} it can be written as
\begin{align}
\fracpar{}{t}\expgf
&=\expgf+\big[(1+u)x+d\big]t\fracpar{}{t}\expgf\nn\\
&+(x-x^2)\fracpar{}{x}\expgf.
\end{align}
Rearrangement of the terms completes the proof.
\end{proof}
\textbf{Acknowledgement}.
This research was supported by the Israel Science Foundation (grants No. 1692/17, No. 1144/16).

\bibliographystyle{plain}

\end{document}